\documentclass[final,1p,times]{elsarticle}

\usepackage[colorlinks=true]{hyperref}
\usepackage{bbm}
\usepackage{amsfonts}
\usepackage{mathrsfs}
\usepackage{amsthm,amsmath,amssymb,mathrsfs,amsfonts,graphicx,graphics,latexsym,exscale,cmmib57,dsfont,amscd}
\usepackage[small,nohug,heads=vee]{diagrams}
\usepackage{color,soul}
\swapnumbers
\newtheorem*{thm*}{\ref{0}.8~Theorem}
\newtheorem*{lem*}{Lemma}
\newtheorem*{1c}{\ref{1}c~Lemma}
\newtheorem*{1d}{\ref{1}d~Lemma}
\newtheorem*{6c}{\ref{6}c~Theorem}
\newtheorem*{4a}{\ref{4}a~Theorem}
\newtheorem*{4b}{\ref{4}b~Theorem}

\newtheorem{thm}{Theorem}
\newtheorem{lem}[thm]{Lemma}

\newtheorem*{7d}{\ref{7}d~Lemma}
\newtheorem*{7e}{\ref{7}e~Theorem}
\newtheorem*{7f}{\ref{7}f~Proposition}
\newtheorem*{7h}{\ref{7}h~Proposition}
\newtheorem*{7i}{\ref{7}i~Proposition}

\newtheorem*{8d}{\ref{8}d~Proposition}
\newtheorem*{8e}{\ref{8}e~Theorem}
\newtheorem*{8f}{\ref{8}f~Theorem}

\theoremstyle{definition}
\newtheorem{sec0}[thm]{Introduction}
\newtheorem{rem}[thm]{Remark}
\newtheorem{se}[thm]{}
\newtheorem*{0.01}{\ref{0}.1}
\newtheorem*{0.02}{\ref{0}.2}
\newtheorem*{0.03}{\ref{0}.3}
\newtheorem*{0.04}{\ref{0}.4}
\newtheorem*{0.05}{\ref{0}.5}
\newtheorem*{0.06}{\ref{0}.6}
\newtheorem*{0.07}{\ref{0}.7}
\newtheorem*{0.09}{\ref{0}.9~Remark}

\newtheorem*{1a}{\ref{1}a}
\newtheorem*{1b}{\ref{1}b}

\newtheorem*{2a}{\ref{2}a}
\newtheorem*{2b}{\ref{2}b}

\newtheorem*{6a}{\ref{6}a}
\newtheorem*{6b}{\ref{6}b}

\newtheorem*{7a}{\ref{7}a}
\newtheorem*{7b}{\ref{7}b}
\newtheorem*{7c}{\ref{7}c}
\newtheorem*{7g}{\ref{7}g}

\newtheorem*{8a}{\ref{8}a}
\newtheorem*{8b}{\ref{8}b}
\newtheorem*{8c}{\ref{8}c}

\newcommand{\K}{\boldsymbol{K}}
\newcommand{\C}{\boldsymbol{C}}

\journal{XXX}

\begin{document}

\begin{frontmatter}
\title{On recurrence in zero-dimensional locally compact flow with compactly generated phase group}

\author{Xiongping Dai}
\ead{xpdai@nju.edu.cn}
\address{Department of Mathematics, Nanjing University, Nanjing 210093, People's Republic of China}

\begin{abstract}
Let $X$ be a zero-dimensional locally compact Hausdorff space not necessarily metric and $G$ a compactly generated topological group not necessarily abelian or countable. We define recurrence at a point for any continuous action of $G$ on $X$, and then, show that if $\overline{Gx}$ is compact for all $x\in X$, the conditions (i) this dynamics is pointwise recurrent, (ii) $X$ is a union of $G$-minimal sets, (iii) the $G$-orbit closure relation is closed in $X\times X$, and (iv) $X\ni x\mapsto \overline{Gx}\in 2^X$ is continuous, are pairwise equivalent. Consequently, if this dynamics is distal, then it is equicontinuous.
\end{abstract}

\begin{keyword}
Recurrence $\cdot$ Distality $\cdot$ Zero dimensional flow $\cdot$ Compactly generated group

\medskip
\MSC[2010] 37B05, 54H15
\end{keyword}
\end{frontmatter}
\begin{sec0}\label{0}
Let $(G,X)$ be a \textit{flow} with phase group $G$ and with phase space $X$, which is in the following sense unless stated otherwise:
\begin{enumerate}[\textcircled{a}]
\item[\textcircled{a}] $G$ is a Hausdorff semi-topological group, namely: $G$ is a group with a Hausdorff topology under which $G\times G\xrightarrow{(s,t)\mapsto st}G$ is jointly continuous but $G\xrightarrow{t\mapsto t^{-1}}G$ need not continuous;
\item[\textcircled{b}] $X$ is a locally compact Hausdorff space; and moreover,
\item[\textcircled{c}] there is an action of $G$ on $X$, denoted $G\times X\xrightarrow{(t,x)\mapsto tx}X$ such that
\begin{enumerate}[(i)]
\item  Continuity: $(t,x)\mapsto tx$ is jointly continuous;
\item Transformation group: $ex=x$ and $(st)x=s(tx)$ for all $x\in X$ and $s,t\in G$, where $e$ is the identity of $G$;
\item Poisson stability: $Gx$ is relatively compact (i.e. $\overline{Gx}$ is compact) in $X$ for all $x\in X$.
\end{enumerate}
\end{enumerate}
By \textcircled{b} there is a compatible uniformity structure, denoted $\mathscr{U}$, on $X$. In the sequel, $\mathfrak{N}_x$ stands for the neighborhood filter at $x$ of $X$.

\begin{0.01}
Let $Y$ be a Hausdorff space and $2^Y$ the collection of non-empty closed subsets of $Y$. For a net $\{Y_i\}$ in $2^Y$ and $K\in 2^Y$, we say `$Y_i\rightharpoonup K$ in $2^Y$' if
$$K=\{y\in Y\,|\,\exists\textrm{ a subnet }\{Y_{i_k}\}\textrm{ from }\{Y_i\}\textrm{ and }y_{i_k}\in Y_{i_k}\textrm{ s.t. }y_{i_k}\to y\}.$$
\begin{lem*}
If $K_i\rightharpoonup K$ in $2^Y$ and $\{K_{i_n}\}$ is any subnet of $\{K_i\}$ with $K_{i_n}\rightharpoonup K^\prime$, then $K^\prime\subseteq K$.
\end{lem*}
If $Y$ is discrete and $\{Y_i\}_{i=1}^\infty$ with $Y_i\rightharpoonup K$, then $K=\bigcap_{n\ge1}\bigcup_{i\ge n}Y_i=\limsup_{i\to\infty}Y_i$; in other words, $K=\{Y_i,\ \textrm{i.o.}\}$.
\end{0.01}

\begin{0.02}
Associated to $(G,X)$ we define the orbit closure mapping $\mathscr{O}_{G,X}\colon X\rightarrow 2^X$ by $x\mapsto \overline{Gx}$. Then we say:
\begin{enumerate}[a)]
\item $\mathscr{O}_{G,X}$ is $\ulcorner$\!continuous\!$\urcorner$, provided that if $x_i\to x$ in $X$ and $\overline{Gx_i}\rightharpoonup\mathfrak{O}$ in $2^X$, then $\overline{Gx}=\mathfrak{O}$.
\item $\mathscr{O}_{G,X}$ is upper semi-continuous, provided that for every $x\in X$ and every neighborhood $V$ of $\overline{Gx}$ there is a $U\in\mathfrak{N}_x$ such that $\overline{Gy}\subseteq V$ for all $y\in U$.
\end{enumerate}
Clearly, if $\mathscr{O}_{G,X}$ is upper semi-continuous, then $\mathscr{O}_{G,X}$ is $\ulcorner$\!continuous\!$\urcorner$. However, no converse is valid in general.
Given $U\subset X$, define 
\begin{enumerate}[c)]
\item $U_G^*=\{x\in X\,|\,\overline{Gx}\subseteq U\}$. 
\end{enumerate}
If $U$ is closed, then $U_G^*=\bigcap_{g\in G}gU$ is closed.
\end{0.02}

\begin{0.03}
The `orbit closure relation $R_o$ of $(G,X)$' is defined by $(x,y)\in R_o$ if $y\in\overline{Gx}$. It is not difficult to prove that
\begin{enumerate}[a)]
\item \textit{$\mathscr{O}_{G,X}$ is $\ulcorner$\!continuous\!$\urcorner$ iff $R_o$ is closed in $X\times X$. (Warning: $R_o$ closed $\not\Rightarrow$ $\mathscr{O}_{G,X}$ upper semi-continuous in general.)}
\end{enumerate}
Indeed, suppose $\mathscr{O}_{G,X}$ is $\ulcorner$\!continuous\!$\urcorner$, and, let $R_o\ni(x_i,y_i)\to(x,y)$. Then $\overline{Gx_i}\rightharpoonup\overline{Gx}$ by \ref{0}.2 and $y_i\in\overline{Gx_i}$ with $y_i\to y$. So $y\in\overline{Gx}$ by \ref{0}.1 and $(x,y)\in R_o$. Thus $R_o$ is closed. Conversely, assume $R_o$ is closed, and, let $x_i\to x$ in $X$ and $\overline{Gx_i}\rightharpoonup\mathfrak{O}$ in $2^X$. If $\overline{Gx}\not=\mathfrak{O}$, then $\overline{Gx}\varsubsetneq\mathfrak{O}$. For $y\in\mathfrak{O}\setminus\overline{Gx}$ there is a net $y_{i_k}\in \overline{Gx_{i_k}}$ with $y_{i_k}\to y$. Since $(x_{i_k},y_{i_k})\in R_o$, so $(x,y)\in R_o$ and $y\in \overline{Gx}$ contrary to $y\in\mathfrak{O}\setminus\overline{Gx}$. Thus $\overline{Gx}=\mathfrak{O}$.
\begin{enumerate}[b)]
\item {\it If $X$ is compact here, then $\mathscr{O}_{G,X}$ is $\ulcorner$\!continuous\!$\urcorner$ iff it is upper semi-continuous iff it is continuous iff $R_o$ is closed.}
\end{enumerate}
\end{0.03}

\begin{0.04}
We say that a set $K$ in $X$ is `minimal' under $(G,X)$ if $\overline{Gx}=K$ for all $x\in K$. In this case $K$ is a closed invariant subset of $X$, and moreover, $K$ is compact by (iii).
\begin{lem*}
Let $R_o$ be the orbit closure relation of $(G,X)$. Suppose $y\in X$ such that $R_o[y]$ is closed: if $(x_n,y)\in R_o$ with $x_n\to x$, then $(x,y)\in R_o$. Then $\overline{Gy}$ is minimal.
\end{lem*}

\begin{proof}
Since $\overline{Gy}$ is compact, there is a point $x\in\overline{Gy}$ such that $\overline{Gx}$ is minimal. There is a net $t_n\in G$ with $t_ny\to x$. Since $(t_ny,y)\in R_o$, so $(x,y)\in R_o$ and $y\in\overline{Gx}$. Thus $\overline{Gy}=\overline{Gx}$ is minimal.
\end{proof}
\end{0.04}

\begin{0.05}
A subset $A$ of $G$ is `syndetic' in $G$ if there exists a compact subset $K$ of $G$ with $G=KA$. A set $B$ is said to be `thick' in $G$ if for all compact set $K$ in $G$ there corresponds an element $t\in B$ such that $Kt\subseteq B$. It is a well-known fact that $A$ is syndetic in $G$ if and only if it intersects non-voidly with every thick subset of $G$.
\begin{enumerate}[a)]
\item A point $x\in X$ is `almost periodic' (a.p) under $(G,X)$ if $N_G(x,U)=\{t\in G\,|\,tx\in U\}$ is a syndetic subset of $G$ for all $U\in\mathfrak{N}_x$.
Since $\overline{Gx}$ is compact by (iii) here, $x\in X$ is a.p if and only if $\overline{Gx}$ is minimal under $(G,X)$ (cf., e.g.~\cite[Note~2.5.1]{AD}). Thus, $N_G(x,U)$ is discretely syndetic in $G$ for every a.p point $x$ of $(G,X)$.

\item Following \cite{GH} we say that:
\begin{enumerate}[1)]
\item $(G,X)$ is `weakly a.p' if given $\alpha\in\mathscr{U}$ there is a finite subset $F$ of $G$ such that $Ftx\cap\alpha[x]\not=\emptyset$ for all $x\in X$ and $t\in G$.
\item $(G,X)$ is `locally weakly a.p' if given $\alpha\in\mathscr{U}$ and $x\in X$ there is a finite subset $F$ of $G$ and some $U\in\mathfrak{N}_x$ such that $Fty\cap\alpha[y]\not=\emptyset$ for all $y\in U$ and $t\in G$.
\end{enumerate}
Clearly, if $(G,X)$ is locally weakly a.p, it is pointwise a.p.
\end{enumerate}
\end{0.05}

\begin{0.06}
A Hausdorff space is said to be \textit{zero dimensional} if it has a base consisting of clopen sets. For example, the coset spaces are locally compact zero-dimensional in a totally disconnected locally compact topological group.

Recall that $G$ is \textit{compactly generated}, provided that there exists a compact subset $\Gamma$ of $G$, called a \textit{generating set}, such that
$G=\bigcup_{r=1}^\infty\Gamma^r$. In this case, $G$ is $\sigma$-compact and we may assume $e\in \Gamma=\Gamma^{-1}$. In this case, the generating set is not unique; e.g., $t^{-1}\Gamma t$, for $t\in G$, is so. 
\end{0.06}

\begin{0.07}
In this note we consider the relationships of the following important dynamics of $(G,X)$:
\begin{enumerate}[(1)]
\item The pointwise ``recurrence'' in certain sense (cf.~Def.~\ref{2}).
\item The pointwise almost periodicity.
\item $X$ is a union of minimal sets.
\item $R_o$ is closed.
\item $\mathscr{O}_{G,X}$ is $\ulcorner$\!continuous\!$\urcorner$.
\item $\mathscr{O}_{G,X}$ is continuous.
\item The local weak almost periodicity.
\end{enumerate}
It is evident that $(6)\Rightarrow(5)\Leftrightarrow(4)\Rightarrow(3)\Leftrightarrow(2)\Rightarrow(1)$. However, even for $G=\mathbb{Z}$ or $\mathbb{R}$ and $X$ a compact metric space, that ``$(1)\Rightarrow(5)$'' is not generally true is shown by very simple examples.

Note that if $(G,X)$ is distal and $(G,X\times X)$ satisfies (5), then it is not difficult to prove that the regionally proximal relation of $(G,X)$ is equal to $\Delta_X$, the diagonal of $X\times X$, so that $(G,X)$ is equicontinuous \cite{E58} (cf.~Theorem~\ref{6}c below). In McMahon and Wu \cite{MW76}, there is an example of a countable group $G$ acting on a zero-dimensional compact $X$ such that $(G,X)$ is distal non-equicontinuous. So in this example, $(G,X\times X)$, with zero-dimensional compact phase space $X\times X$ and with countable (so $\sigma$-compact) phase group $G$, satisfies that $(1)\not\Rightarrow(5)$ and $(3)\not\Rightarrow(5)$. Thus we need impose some more restrictive condition on $G$ for $(1)\textrm{ or }(3)\Rightarrow(5)$.

Let $X$ be a zero-dimensional compact space. In \cite[Theorems~5 and 6]{G46}, Gottschalk shows $(1)\Rightarrow(5)\Leftrightarrow(7)$ for $\mathbb{Z}_+$-flow on $X$, and see also \cite[Theorem~7.12]{AD} for $\mathbb{Z}$-flow on $X$.

Recall that $G$ is \textit{generative} provided that $G$ is abelian and generated by some compact neighborhood of $e$ (cf.~\cite[Def.~6.01]{GH}). Using ``replete semigroup'', in \cite{GH} Gottschalk and Hedlund formulate a definition of recurrence. Then, with this notion, they prove that $(1) \Rightarrow (5)\Leftrightarrow(7)$ for each flow with generative phase group and with zero-dimensional phase space (cf.~\cite[Theorems~7.07 and 7.08]{GH}).

In \cite{AGW}, Auslander et al. introduce a definition of recurrence in terms of ``cone''. With their notion of recurrence, they show that $(1) \Rightarrow (5)$ for every flow $(G,X)$ with finitely generated phase group $G$ not necessarily abelian and with zero-dimensional compact metric phase space $X$ (cf.~\cite[Theorem~1.8]{AGW}).
\end{0.07}

The purpose of this note is to formulate a recurrence and extend the theorem of Auslander et al. \cite[Theorem~1.8]{AGW} to any flow $(G,X)$ with compactly generated phase group $G$ not necessarily countable or abelian, and, with zero-dimensional phase space $X$ not necessarily compact metrizable, as follows:

\begin{thm*}[see Thm.~\ref{4}a]
Let $(G,X)$ be a flow, where $G$ is compactly generated and $X$ is zero dimensional. Then the following conditions are pairwise equivalent:
\begin{enumerate}[(1)]
\item The pointwise ``recurrence'' in certain sense (cf.~Def.~\ref{2}b).
\item The pointwise almost periodicity.
\item $X$ is a union of minimal sets.
\item $R_o$ is closed.
\item $\mathscr{O}_{G,X}$ is $\ulcorner$\!continuous\!$\urcorner$.
\item $\mathscr{O}_{G,X}$ is continuous.
\item For every compact open subset $U$ of $X$, $U_G^*$ is compact open in $X$.
\item $(G,X)$ is locally weakly a.p.
\end{enumerate}
\end{thm*}

Comparing with the case $X$ compact, we have to face the essential point that a net in $X$ need not have a convergent subnet. In addition, comparing with the case $G=\mathbb{Z}_+$, for any point $x$ and any closed subset $U$ of $X$ we have no the ``first visit time'' $t\in G$ with $tx\in U$.

\begin{0.09}
It should be mentioned that a part of Theorem~\ref{0}.8 has been proved by Reid \cite{R2} under the framework that $G$ is ``equicontinuously generated'' by a subset $S$ of $\textrm{Homeo}\,(X)$; that is, $G=\bigcup_{r\ge1} S^r$, where $\textrm{id}_X\in S=S^{-1}$, and $S$ is equicontinuous (cf.~Def.~\ref{6}b). However, our approaches are different completely with and much simpler than Reid \cite[Theorems~1.2, 1.3]{R2} of using very technical arguments of compact-open invariant sets. It turns out that Theorem~\ref{0}.8 concisely proved here implies the ``equicontinuously generated'' case of Reid~\cite{R2} (see Theorems~\ref{8}e and \ref{8}f for alternative simple proofs of Reid's results).
\end{0.09}
\end{sec0}

\begin{se}[Word length and cone]\label{1}
Let $G$ be compactly generated with a generating set $\Gamma$ such that $e\in\Gamma=\Gamma^{-1}$ as in Def.~\ref{0}.6.

\begin{1a}
For any $g\in G$ with $g\not=e$, the \textit{$\Gamma$-length of} $g$, denoted $|g|$, is the smallest integer $r$ such that $g=\gamma_1\dotsm\gamma_r$ with $\gamma_i\in\Gamma$ for $1\le i\le r$, and, write $\K(g)=\Gamma^{|g|-1}\cdot g$, where $\Gamma^0=\{e\}$ and $\K(g)$ is a compact subset of $G$. Clearly, $e\notin\K(g)$ for $g\in G$ with $|g|\ge1$.
\end{1a}

\begin{1b}
A subset $\C$ of $G$ is called a \textit{$\Gamma$-cone} in $G$ if there exists a net $\{g_i\}$ in $G$ such that $|g_i|\nearrow\infty$ and $\K(g_i)\rightharpoonup\C$ in $2^G$. Here $|g_i|\nearrow\infty$ means that if $i\le i^\prime$ then $|g_i|\le|g_{i^\prime}|$.
Clearly, every cone in $G$ is a closed set in $G$. Moreover, if $G$ is discrete (i.e. finitely generated), then $e\notin\C$.
\end{1b}

\begin{1c}[{cf.~\cite[Prop.~1.2 and Prop.~1.5]{AGW} for $G$ finitely generated}]
Let $G$ be compactly generated with a generating set $\Gamma$. Then:
\begin{enumerate}[$(1)$]
\item If $\{g_i\}$ is a net in $G$ with $|g_i|\nearrow\infty$, then $\K(g_i)\rightharpoonup\C\not=\emptyset$.
\item If $\C$ is a $\Gamma$-cone in $G$, then $\C$ is a discretely thick subset of $G$.
\end{enumerate}
\end{1c}

\begin{proof}
Let $\{g_i\}$ be a net in $G$ with $|g_i|\nearrow\infty$ and $\K(g_i)\rightharpoonup\C$. Let $K$ be a finite subset of $G$ with $e\in K$. We shall find an element $t\in G$ with $Kt\subseteq\C$. Since $G$ is compactly generated with the generating set $\Gamma$, there is an integer $r\ge1$ such that $K\subseteq\Gamma^r$. For each $i$, let $\ell_i=|g_i|$ ($\nearrow\infty$) and write
$$
g_i=\gamma_1\dotsm\gamma_{\ell_i-r-1}\gamma_{\ell_i-r}\dotsm\gamma_{\ell_i}\in\Gamma^{\ell_i},\quad\alpha_i=\gamma_{\ell_i-r}\dotsm\gamma_{\ell_i}\in\Gamma^{r+1}.
$$
Then $K\alpha_i\subseteq\K(g_i)$ for $i$ with $\ell_i>r$. Since $\Gamma^{r+1}$ is compact, we can assume (a subnet of) $\alpha_i\to t$ in $G$. Thus $t\in\C$ and $Kt\subseteq\C\not=\emptyset$ by Lemma~\ref{0}.1. This also proves that $\C$ is a discretely thick subset of $G$. The proof is completed.
\end{proof}

From the choice of $\alpha_i$ in the above proof of Lemma~\ref{1}c we can easily conclude the following lemma:

\begin{1d}[{cf.~\cite[Lem.~3.6]{R2}}]
Let $F$ be any finite subset of $G$. Then there is an integer $n>1$ such that for all $g\in G$ with $|g|\ge n$, there is an element $t\in G$ such that $|t|=n$ and $Ft\subseteq\K(g)$.
\end{1d}
\end{se}

\begin{se}[Recurrence]\label{2}
Let $(G,X)$ be a flow with a compactly generated phase group $G$ and $x\in X$. We will define recurrence at $x$ in two ways.
\begin{2a}
Let $\Gamma$ be a generating set of $G$. We say that $x$ is \textit{$\Gamma$-recurrent of type I} under $(G,X)$ if $\C x\cap U\not=\emptyset$ for all $\Gamma$-cone $\C$ in $G$ and every $U\in\mathfrak{N}_x$. See \cite[Def.~1.6]{AGW} for $G$ finitely generated and $X$ a compact metric space in terms of sequence.

We say that $x$ is \textit{recurrent of type I} under $(G,X)$ if $x$ is $\Gamma$-recurrent of type I for all generating set $\Gamma$ of $G$. Clearly, if $x$ is recurrent of type I, then $tx$ is recurrent of type I for every $t\in G$.
\end{2a}

\begin{2b}
Let $\Gamma$ be a generating set of $G$. We say that $x$ is \textit{$\Gamma$-recurrent of type II} under $(G,X)$ if for each $U\in\mathfrak{N}_x$ and every net $\{g_i\}$ in $G$ with $|g_i|\to\infty$ there exists an integer $n\ge1$ such that there is a subnet $\{g_{i_j}\}$ from $\{g_i\}$ and $c_j\in\K(g_{i_j})$ with $c_jx\in U$ and $|c_j|\le n$. See \cite[Def.~1.1]{R2}.
\end{2b}
\end{se}

\begin{lem}\label{3}
Let $(G,X)$ be a flow, where $G$ is compactly generated with a generating set $\Gamma$, and, let $x\in X$. Then:
\begin{enumerate}[$(1)$]
\item If $x$ is an a.p point, then $x$ is $\Gamma$-recurrent of type II (cf.~\cite[Prop.~3.7(iv)]{R2}).
\item If $x$ is $\Gamma$-recurrent of type II, then $x$ is $\Gamma$-recurrent of type I.
\end{enumerate}
\end{lem}

\begin{proof}
(1). Let $U\in\mathfrak{N}_x$ and simply write $A=N_G(x,U)$. Then $A$ is a discretely syndetic subset of $G$ by \ref{0}.5-a). So there is a finite set $F$ in $G$ such that $F^{-1}A=G$, or equivalently, $Ft\cap A\not=\emptyset$ for all $t\in G$. Let $n_1=\max\{|f|+1\,|\,f\in F\}$. Let $\{g_i\}$ be a net in $G$ with $|g_i|\to\infty$.
By Lemma~\ref{1}d, there exists an integer $n>n_1$ such that there is a subnet $\{g_{i_j}\}$ from $\{g_i\}$ such that for each $j$, there is an element $t_j\in G$ with $Ft_j\subseteq\K(g_{i_j})$ and $|t_j|=n$. Take $f_j\in F$ with $c_j=f_jt_j\in A$, then $|c_j|\le2n$ and $c_jx\in U$ for all $j$. Thus $x$ is $\Gamma$-recurrent of type II.

(2). Suppose $x$ is $\Gamma$-recurrent of type II. Let $\C$ be a cone in $G$. That is, $\K(g_i)\rightharpoonup\C$, where $\{g_i\}$ is a net in $G$ with $|g_i|\nearrow\infty$. Let $U, V\in\mathfrak{N}_x$ with $\overline{V}\subseteq U$. By Def.~\ref{2}b, there is an integer $n\ge1$ and a subnet $\{g_{i_j}\}$ from $\{g_i\}$ and $c_j\in\K(g_{i_j})$ such that $c_jx\in V$ and $|c_j|\le n$. Since $\Gamma^n$ is compact and $c_j\in\Gamma^n$, we can assume (a subnet of) $c_j\to c$ in $G$ and $c_jx\to cx\in\overline{V}$. Since $c\in\C$ by definition and $\C x\cap U\not=\emptyset$, $x$ is $\Gamma$-recurrent of type I.
The proof is complete.
\end{proof}

\begin{se}\label{4}
We are ready to state and concisely prove our main theorems of this note self-closely.
Recall that $\mathscr{O}_{G,X}\colon X\rightarrow 2^X$ is defined by $x\mapsto\overline{Gx}$ and $U_G^*=\{x\in X\,|\,\overline{Gx}\subseteq U\}$ for all subset $U$ of $X$.

\begin{4a}
Let $(G,X)$ be any flow, where $G$ is compactly generated with a generating set $\Gamma$ and $X$ is zero-dimensional. Then the following conditions are pairwise equivalent:
\begin{enumerate}[(1)]
\item $(G,X)$ is pointwise $\Gamma$-recurrent of type I.
\item $(G,X)$ is pointwise $\Gamma$-recurrent of type II.
\item $(G,X)$ is pointwise a.p.
\item $X$ is a union of minimal sets.
\item The $G$-orbit closure relation $R_o$ is closed.
\item $\mathscr{O}_{G,X}$ is $\ulcorner\!\textrm{continuous}\!\urcorner$.
\item $\mathscr{O}_{G,X}$ is upper semi-continuous.
\item Given any compact-open subset $U$ of $X$, $U_G^*$ is compact open.
\item $(G,X)$ is locally weakly a.p.
\end{enumerate}
\end{4a}

\begin{proof}
(2)$\Rightarrow$(1): By (2) of Lemma~\ref{3}.

(3)$\Rightarrow$(2): By (1) of Lemma~\ref{3}.

(4)$\Rightarrow$(3): By a) of \ref{0}.5.

(5)$\Rightarrow$(4): By Lemma~\ref{0}.4.

(6)$\Rightarrow$(5): By a) of \ref{0}.3.

(7)$\Rightarrow$(6): By Def.~\ref{0}.2.

(8)$\Rightarrow$(7): Assume (8). Let $x\in X$ and $V$ an open neighborhood of $\overline{Gx}$. Because $X$ is locally compact zero-dimensional and $\overline{Gx}$ is compact, there is a compact-open set $U$ with $\overline{Gx}\subseteq U\subseteq V$. Then $x\in\overline{Gx}\subseteq U_G^*\subseteq V$. Since $U_G^*$ is open, $\mathscr{O}_{G,X}$ is upper semi-continuous. Thus (8)$\Rightarrow$(7) holds.

(1)$\Rightarrow$(8): Assume (1). Let $U$ be a compact-open subset of $X$. Then $U_G^*$ is an $G$-invariant compact set. To prove that $U_G^*$ is open, suppose the contrary that $U_G^*$ is not open. Then $U_G^*\not=\emptyset$ and there is a net $x_i\in U\setminus U_G^*$ and a point $x\in U_G^*$ with $x_i\to x$. We can obviously take $g_i\in G$ such that
$g_i x_i\notin U$ and $|g_i|=\min\{|t|\colon t\in G, tx\notin U\}$.
Since $U$ is clopen and $\overline{Gx}\subseteq U_G^*$, obviously (a subnet of) $|g_i|\to\infty$. For all $i$ we write
$g_i=\gamma_{i,1}\gamma_{i,2}\dotsm\gamma_{i,|g_i|}$ and $g_i^\prime=\gamma_{i,2}\dotsm\gamma_{i,|g_i|}$.
Then $g_i^\prime x_i\in U$ and $\gamma_{i,1}\in\Gamma$ for all $i$. Since $U$ and $\Gamma$ are compact, we can assume (a subnet of) $\gamma_{i,1}\to\gamma$ and $g_i^\prime x_i\to y^\prime\in U$. Let $y=\gamma y^\prime$. Then $y=\lim_ig_i x_i\notin U$. Let $\K(g_i^{-1})\rightharpoonup\C$, a $\Gamma$-cone in $G$. Given $c\in \C$, we can take a subnet $\{g_{i_j}\}$ from $\{g_i\}$ and $c_j=k_jg_{i_j}^{-1}$ with $k_j\in\Gamma^{|g_{i_j}|-1}$ and $c_j\to c$. Thus by $|k_j|<|g_{i_j}|$, it follows that
$cy={\lim}_jc_jy={\lim}_j (k_jg_{i_j}^{-1})(g_{i_j}x_{i_j})\in U$. So $\C y\cap(X\setminus U)=\emptyset$, contrary to that $y$ is $\Gamma$-recurrent of type I by (1). Thus (1) implies (8).

(9)$\Rightarrow$(3): By b) of \ref{0}.5.

(7)$\Rightarrow$(9): Let $\alpha\in\mathscr{U}$ and $x\in X$. Let $V\in\mathfrak{N}_x$ be so small that $\alpha[y]\supseteq V$ for all $y\in V$. Since $\overline{Gx}$ is compact minimal, there exists a finite set $\{z_1,\dotsc,z_k\}\subset X$ with $W_i\in\mathfrak{N}_{z_i}$, $t_i\in G$ such that $t_iW_i\subseteq V$ and $\overline{Gx}\subseteq\bigcup_{i=1}^kW_i$. By (7), there is a $U\in\mathfrak{N}_x$ with $U\subseteq V$ such that $\overline{Gy}\subseteq\bigcup_{i=1}^kW_i$ for all $y\in U$. Then for $F=\{t_1,\dotsc,t_k\}$, $Fty\cap\alpha[y]\not=\emptyset$ for all $y\in U$ and $t\in T$. Thus (7) implies (9).

The proof of Theorem~\ref{4}a is thus completed.
\end{proof}

\begin{4b}
Let $(G,X)$ be any flow with $G$ compactly generated with a generating set $\Gamma$ and with $X$ zero-dimensional, where $\overline{Gx}$ need not be compact for $x\in X$.
Let $V$ be an open subset of $X$ and write $V_\infty=\bigcap_{g\in G}gV$. Suppose $V_\infty$ is compact. Then
$V_\infty$ is open if and only if $\overline{Gx}\cap V_\infty=\emptyset$ for every $x\in X\setminus V_\infty$.
\end{4b}

\begin{proof}
Necessity is obvious for $X\setminus V_\infty$ is closed $G$-invariant. To show sufficiency, suppose the contrary that $V_\infty$ is not open (so $V_\infty\not=\emptyset$). Then there is a net $x_i\in V\setminus V_\infty$ with $x_i\to x\in V_\infty$. Since $X$ is a locally compact zero-dimensional Hausdorff space, we can find a compact-open neighborhood $U$ of $V_\infty$ with $U\subseteq V$. Clearly, $V_\infty=\bigcap_{g\in G}gU$. We may assume $\{x_i\}\subset U$. Take $g_i\in G$ such that $g_ix_i\notin U$ and $|g_i|=\min\{|t|\colon t\in G, tx_i\notin U\}$. Since $V_\infty$ is compact $G$-invariant, $|g_i|\to\infty$. By compactness of $U$, we can assume (a subnet of) $g_ix_i\to y\notin U$ (cf.~Proof of (1)$\Rightarrow$(8) in Theorem~\ref{4}a). Let $\K(g_i^{-1})\rightharpoonup\C$. Then $\C y\subseteq U$ (cf. Proof of (1)$\Rightarrow$(8) in Theorem~\ref{4}a). Next we shall prove that $\overline{\C y}$ contains an $G$-orbit so $\overline{Gy}\cap V_\infty\not=\emptyset$, a contradiction.

Indeed, let $\mathcal {F}$ be the collection of finite subsets of $G$ with a partial order by inclusion. Since $\C$ is discretely thick in $G$ by Lemma~\ref{1}c, we have for $F\in\mathcal {F}$ that there exists an element $t_F\in\C$ with $Ft_F\subseteq\C$. Since $2^U$ is a compact Hausdorff space and $t_Fy\in U$ and $Ft_Fy\in 2^U$, we can assume (a subnet of) $Ft_Fy\to Q\subseteq\overline{\C y}$ and $t_Fy\to y^\prime$. Given $t\in G$ there is some $F_0\in\mathcal {F}$ with $t\in F_0$. Then $t\in F$ for all $F\ge F_0$ so that $ty^\prime=\lim_Ftt_Fy\in\lim_FFt_Fy$. Thus $ty^\prime\in\overline{\C y}$ for all $t\in G$. Then $\overline{Gy^\prime}\subseteq\overline{\C y}$.
The proof is completed.
\end{proof}
\end{se}

\begin{rem}
It should be noticed that since we do not know whether the $\Gamma$-recurrence of type I at a point of $X$ is equivalent to that of type II, so $(1)\Leftrightarrow(2)$ in Theorem~\ref{4}a is not trivial. In addition, Theorem~\ref{4}b can be utilized for proving $(5)\Rightarrow(8)$ in Theorem~\ref{4}a.
\end{rem}

\begin{se}\label{6}
Let $(G,X)$ be a flow.
\begin{6a}
We say $(G,X)$ is distal if we have for $x,y\in X$ with $x\not=y$ that $\Delta_X\cap\overline{G(x,y)}=\emptyset$.
\end{6a}

\begin{6b}
We say $(G,X)$ is equicontinuous if given $\varepsilon\in\mathscr{U}$ and $x\in X$ there exists an index $\delta\in\mathscr{U}$ such that $g\delta[x]\subseteq\varepsilon[gx]$ for all $g\in G$. If $\delta$ is independent of $x\in X$, then we say $(G,X)$ is uniformly equicontinuous in this case.
\end{6b}

Using Theorem~\ref{4}a we can then conclude the following generalization of an important theorem of R.~Ellis:

\begin{6c}[{cf.~\cite[Thm.~2]{E58} for $G$ generative and \cite[Cor.~1.9]{AGW} for $G$ finitely generated with $X$ compact}]
Let $(G,X)$ be a distal flow with $X$ zero-dimensional and $G$ compactly generated. Then $(G,X)$ is equicontinuous.
\end{6c}

\begin{proof}
Since $(G,X)$ is distal, $(G,X\times X)$ is pointwise a.p. To show $(G,X)$ equicontinuous, suppose the contrary that $(G,X)$ is not equicontinuous; then there is a point $x\in X$, an index $\varepsilon\in\mathscr{U}$, a net $x_n\to x$ and a net $t_n\in G$ such that $(t_nx,t_nx_n)\notin\varepsilon$. By (7) of Theorem~\ref{4}a, we can assume $t_nx\to x_*$ and $t_nx_n\to y_*$ with $(x_*,y_*)\notin\varepsilon$. Let $s_n=t_n^{-1}$ and $x_n^\prime=t_nx$, $y_n^\prime=t_nx_n$. Then $(s_nx_n^\prime,s_ny_n^\prime)\to(x,x)$ and $(x_n^\prime,y_n^\prime)\to(x_*,y_*)$. Using (7) of Theorem~\ref{4}a for $(G,X\times X)$, it follows that $x_*=y_*$, contrary to $(x_*,y_*)\notin\varepsilon$. The proof is complete.
\end{proof}
\end{se}

\begin{se}[{The case $G$ to be finitely generated \cite{AGW}}]\label{7}
Let $G$ be finitely generated with a generator set $\Gamma$ such that $\Gamma=\Gamma^{-1}$, and let $(G,X)$ be a flow. We shall show that our $\Gamma$-recurrence of type I (Def.~\ref{2}a) is consistent with \cite{AGW} in this case. See Lemma~\ref{7}d below.

\begin{7a}
Let $B_r=\{g\in G\,|\,g\not=e, |g|\le r\}$, for all $r\ge1$. For $g\in G$ with $g\not=e$, write $K(g)=B_{|g|-1}\cdot g$.
\end{7a}

\begin{7b}
A subset $C$ of $G$ is an AGW-\textit{cone}, namely, a cone in the sense of Auslander, Glasner and Weiss~\cite{AGW}, if there is a sequence $g_n\in G$ with $|g_n|\to\infty$ such that for each $r\ge1$ there exists $n_r$
such that $B_r\cap K(g_n)$ is independent of $n$ for all $n\ge n_r$, and, $C=\lim_{n\to\infty}K(g_n)$. Since $G$ is discrete here, so $c\in C$ iff $c\in K(g_n)$ as $n$ sufficiently large. Moreover, $e\notin C$.
\end{7b}

\begin{7c}[{cf.~\cite[Def.~1.6]{AGW}}]
We say that $x\in X$ is AGW-\textit{recurrent}, if $Cx\cap U\not=\emptyset$ for every $U\in\mathfrak{N}_x$ and every AGW-cone $C$ in $G$.
\end{7c}

\begin{7d}
Let $X$ be a compact metric space and $x\in X$. Then $x$ is AGW-recurrent under $(G,X)$ if and only if $x$ is $\Gamma$-recurrent of type I in the sense of Def.~\ref{2}a.
\end{7d}

\begin{proof}
1). Let $\C$ be a $\Gamma$-cone in $G$ with $\K(g_i)=\Gamma^{|g_i|-1}\cdot g_i\rightharpoonup\C$, where $\{g_i\}$ is a net in $G$ such that $|g_i|\nearrow\infty$.
For all integer $k\ge1$, there is an $i_k\in\{i\}$ such that $|g_i|\ge k$ for all $i\ge i_k$ and $|g_{i_k}|<|g_{i_{k+1}}|$.
Clearly $|g_{i_k}|\to\infty$ as $k\to\infty$. Moreover, the sequence $\{g_{i_k}\}_{k=1}^\infty$ is a subnet of $\{g_i\}$. Indeed, for $i^\prime\in\{i\}$, there is an integer $k\ge1$ with $|g_{i^\prime}|<k\le |g_{i_{k^\prime}}|$ for all $k^\prime>k$ so $i^\prime\le i_{k^\prime}$ for all $k^\prime>k$.
Therefore, using $\{i_k\}$ in place of $\{i\}$ if necessary, we can assume $\{g_i\}_{i=1}^\infty$ is a sequence in $G$ with $|g_i|\nearrow\infty$ as $i\to\infty$.

Each $B_r$ is finite, so we may choose a subsequence $\{i_k\}$ from $\{i\}$ so that each $K(g_{i_k})\cap B_r$ is eventually constant. Then by a diagonal process we can choose a subsequence (and relabel) $\{i_k\}_{k=1}^\infty$ from $\{i\}$ such that $K(g_{i_k})\to C$ as $k\to\infty$. Then $C$ is an AGW-cone in $G$ such that $C\subseteq\C$ by Lemma~\ref{0}.1.

2). Let $C$ be a cone in $G$ with $C=\lim_iK(g_i)$ in the sense of Def.~\ref{7}b, where $\{g_i\}_{i=1}^\infty$ is a sequence in $G$ such that $|g_i|\to\infty$ and for each $r\ge1$ there exists an $n_r\ge1$ such that $B_r\cap K(g_i)$ is independent of $i$ for all $i\ge n_r$. By induction, we may require $n_r<n_{r+1}$ and $|g_{n_r}|<|g_{n_{r+1}}|$ for $r\ge1$. Then
$$
C={\bigcup}_{r=1}^\infty(B_r\cap K(g_{n_r})).
$$
Let $\K(g_{n_r})\rightharpoonup\C$ in the sense of Def.~\ref{1}b. Let $c\in\C$. Then there exists a subnet $\{n_{r^\prime}\}$ from $\{n_r\}$ and $k_{n_{r^\prime}}\in\K(g_{n_{r^\prime}})$ such that $k_{n_{r^\prime}}\to c$. Since $G$ is discrete and $\{c\}$ is an open neighborhood of $c$ in $G$, we can assume $k_{n_{r^\prime}}=c$ for all $r^\prime$. Since we have for all $n_r$ that $n_{r^\prime}\ge n_r$ as $r^\prime$ sufficiently large, so $c\in C$ and $\C\subseteq C$.

The above discussion of 1) and 2) implies that $x$ is AGW-recurrent iff $x$ is also $\Gamma$-recurrent of type I. The proof is complete.
\end{proof}

\begin{7e}[{cf.~\cite[Thm.~1.8]{AGW}}]
Let $(G,X)$ be a flow, where $G$ is finitely generated and $X$ is a zero-dimensional compact metric space. Then the following conditions are equivalent:
\begin{enumerate}[(i)]
\item $(G,X)$ is pointwise AGW-recurrent.
\item $(G,X)$ is pointwise a.p.
\item The $G$-orbit closure relation is closed.
\end{enumerate}
\end{7e}

\begin{proof}
In view of Lemma~\ref{7}d, this is a special case of Theorem~\ref{4}a.
\end{proof}

\begin{7f}
Let $f\colon X\rightarrow X$ be a homeomorphism of $X$, which is thought of as a $\mathbb{Z}$-flow. Then $x\in X$ is recurrent of type I (cf.~Def.~\ref{2}a) iff there exists a net $\{i_n\}$ in $\mathbb{Z}$ with $i_n\to\infty$ and a net $\{j_n\}$ in $\mathbb{Z}$ with $j_n\to-\infty$ such that $f^{i_n}x\to x$ and $f^{j_n}x\to x$ simultaneously.
\end{7f}

\begin{proof}
Note that a $\Gamma$-cone $\C$ in $\mathbb{Z}$ with $\Gamma=\{-1,0,1\}$ is either $\C=\mathbb{N}$ or $\C=-\mathbb{N}$, where $\mathbb{N}$ is the set of positive integers. Then the statement follows easily from this fact.
\end{proof}

\begin{7g}[{Regular almost periodicity~\cite{GH}}]
A point $x$ is said to be \textit{regularly a.p} under $(G,X)$ if $N_G(x,U)$ contains a syndetic subgroup of $G$ for every $U\in\mathfrak{N}_x$. We say $(G,X)$ is \textit{regularly a.p} if every $\alpha\in\mathscr{U}$ there is a syndetic subgroup $A$ of $G$ such that $Ax\subseteq\alpha[x]$ for all $x\in X$.
\end{7g}

\begin{7h}[{cf.~\cite[Thm.~7.10]{GH} for $G$ abelian finitely generated}]
Let $(G,X)$ be a flow with $G$ finitely generated and with $X$ zero-dimensional compact. If $(G,X)$ is pointwise regularly a.p, then $(G,X)$ is regularly a.p.
\end{7h}

\begin{proof}
By \cite[Lemma~5.04]{GH}, $(G,X)$ is distal. Further by Corollary~\ref{6}c, it follows that $(G,X)$ is equicontinuous. Finally by \cite[Remark~5.02, Theorem~5.17]{GH}, we see that $(G,X)$ is regularly a.p. The proof is complete.
\end{proof}

\begin{7i}[{cf.~\cite[Thm.~7.11]{GH} for $G$ abelian finitely generated}]
Let $(G,X)$ be a flow with $G$ finitely generated and with $X$ a compact metric space. Then $(G,X)$ is regularly a.p if and only if $(G,X)$ is pointwise regularly a.p and weakly a.p.
\end{7i}

\begin{proof}
The proof of \cite[Thm.~7.11]{GH} is still valid for this case using Theorem~\ref{4}a in place of \cite[Theorem~7.08]{GH}. We omit the details here.
\end{proof}
\end{se}

\begin{se}[The case $G$ to be equicontinuously generated]\label{8}
Let $X$ be a locally compact Hausdorff space, $C_u(X,X)$ the space of continuous mappings from $X$ to itself equipped with the topology of uniformly convergence, $\textrm{Homeo}\,(X)$ the group of self-homeomorphisms of $X$ onto itself as a subspace of $C_u(X,X)$. Clearly, $\textrm{Homeo}\,(X)$ is a Hausdorff semi-topological group so that $(\textrm{Homeo}\,(X),X)$ defined by $\textrm{Homeo}\,(X)\times X\xrightarrow{(f,x)\mapsto fx}X$ is a flow.

In the sequel let $S\subset\textrm{Homeo}\,(X)$ such that: 1) $e=\textrm{id}_X\in S=S^{-1}$ and 2) $S$ is equicontinuous on $X$ (cf.~Def.~\ref{6}b).

\begin{8a}
Set $\langle S\rangle=\bigcup_{r=1}^\infty S^r$. Clearly $\langle S\rangle$ is a subgroup of $\textrm{Homeo}\,(X)$, which is said to be \textit{equicontinuously generated with the generating set $S$}.
\end{8a}

\begin{8b}
Write $\Gamma=\textrm{cls}_uS$, where $\textrm{cls}_u$ denotes the closure relative to $C_u(X,X)$. Then:

\begin{lem*}
$\Gamma$ is a compact subset of $C_u(X,X)$ and $\Gamma\subset\textrm{Homeo}\,(X)$ such that $e\in\Gamma=\Gamma^{-1}$.
\end{lem*}

\begin{proof}
Since $S$ is equicontinuous, so is $\Gamma$. Then $\Gamma$ is compact in $C_u(X,X)$ by Ascoli's Theorem (cf.~\cite[Theorem~7.17]{K}). Moreover, $\Gamma x$ is compact in $X$. We need prove that $\Gamma\subset\textrm{Homeo}\,(X)$. For this, let $S\ni f_n\to f\in\Gamma$. 

Let $x\not=y$ and $\varepsilon\in\mathscr{U}$ with $(x,y)\notin\varepsilon$. If $fx=fy=z$, then for $\delta\in\mathscr{U}$ we have $x, y\in f_n^{-1}\delta[z]$ as $n$ sufficiently large, contrary to equicontinuity of $S$. Thus $f$ is injective.
Let $x\in X$. Since $f_nX=X$, there is a point $x_n\in X$ with $f_nx_n=x$ or $x_n=f_n^{-1}x$. We can assume (a subnet of) $x_n\to x^\prime$. Then $fx^\prime=x$ so $fX=X$. Thus $f$ is a continuous bijection.

Let (a subnet of) $f_n^{-1}\in S\to g\in\Gamma$. Since $e=f_nf_n^{-1}\to fg$, so $fg=e$, $g=f^{-1}\in\Gamma$, and  $f\in\textrm{Homeo}\,(X)$. The proof is complete.
\end{proof}
\end{8b}

\begin{8c}
Set $\langle\Gamma\rangle=\bigcup_{r=1}^\infty \Gamma^r$. Then $\langle\Gamma\rangle$ is a compactly generated subgroup of $\textrm{Homeo}\,(X)$ with a generating set $\Gamma$. Moreover, $(\langle\Gamma\rangle,X)$ is a flow. It is evident that $S^r$ is dense in $\Gamma^r$ for all $r\ge1$ and $\langle S\rangle$ is dense in $\langle\Gamma\rangle$ under the topology of uniform convergence.
\end{8c}

Following Def.~\ref{2}b, $x\in X$ is $S$-recurrent of type II under $(\langle S\rangle,X)$ if for every $U\in\mathfrak{N}_x$ and every net $\{g_i\}$ in $\langle S\rangle$ with $|g_i|\to\infty$, there is an integer $n$, a subnet $\{g_{i_j}\}$ from $\{g_i\}$ and $c_j\in S^{|g_{i_j}|-1}\cdot g_{i_j}$ such that $c_jx\in U$ and $|c_j|\le n$. We do not know if a point $S$-recurrent of type II is $\Gamma$-recurrent of type II, yet it is $\Gamma$-recurrent of type I by Lemma~\ref{8}d below.

\begin{8d}
If $x$ is $S$-recurrent of type II under $(\langle S\rangle,X)$, then $x$ is $\Gamma$-recurrent of type I under $(\langle\Gamma\rangle,X)$.
\end{8d}

\begin{proof}
Let $\{g_i\}$ be a net in $\langle\Gamma\rangle$ with $|g_i|\nearrow\infty$ and $\K(g_i)\rightharpoonup\C$. Let $U, V\in\mathfrak{N}_x$ with $\overline{V}\subseteq U$. We need prove that $\C x\cap U\not=\emptyset$. Since $\langle S\rangle$ is dense in $\langle\Gamma\rangle$, we can assume $g_i\in\langle S\rangle$ for all $i$. Further by definition, there is an integer $n\ge1$, a subnet $\{g_{i_j}\}$ from $\{g_i\}$ and $c_j\in S^{|g_{i_j}|-1}\cdot g_{i_j}\subseteq\K(g_{i_j})$ such that $c_jx\in V$ and $|c_j|\le n$. By $c_j\in\Gamma^n$ and $\Gamma^n$ is compact, it follows that (a subnet of) $c_j\to c\in\C$ and $cx\in U$. The proof is complete.
\end{proof}

\begin{8e}[{cf.~\cite[Thm.~2.2]{R1}, \cite[Thm.~1.2]{R2} using compact-open invariant sets}]
Let $X$ be zero-dimensional, let $G$ be an equicontinuously generated group. Let $V$ be an open subset of $X$ and $V_\infty=\bigcap_{g\in G}gV$. Suppose $V_\infty$ is compact. Then $V_\infty$ is open if and only if $\overline{Gx}\cap V_\infty=\emptyset$ for every $x\in X\setminus V_\infty$.
\end{8e}

\begin{proof}
Let $G=\langle S\rangle$, where $S$ is the equicontinuously generating set of $G$. Then $\overline{Gx}=\overline{\langle\Gamma\rangle x}$ for all $x\in X$. Since $V_\infty$ is compact closed, so $V_\infty=\bigcap_{g\in\langle\Gamma\rangle}gV$. Then Theorem~\ref{8}e follows from Theorem~\ref{4}b.
\end{proof}

\begin{8f}[{cf.~\cite[Thm.~1.3]{R2} using Thm.~\ref{8}e}]
Let $X$ be zero-dimensional and $G=\langle S\rangle$ such that $\overline{Gx}$ is compact for all $x\in X$. Then the following are pairwise equivalent:
\begin{enumerate}[i)]
\item $(G,X)$ is pointwise $S$-recurrent of type II.
\item $(G,X)$ is pointwise a.p.
\item Given any compact open subset $V$ of $X$, $V_G^*$ is open.
\item The orbit closure relation $R_o$ of $(G,X)$ is closed.
\item There exists an extension $\rho\colon(G,X)\rightarrow(G,Y)$ such that: 
\begin{enumerate}[1)]
\item $Y$ is locally compact zero-dimensional, 
\item $gy=y$ for all $y\in Y$ and $g\in G$, and 
\item $\rho^{-1}y$ is a minimal subset of $(G,X)$ for each $y\in Y$.
\end{enumerate}
\end{enumerate}
\end{8f}

\begin{proof}
By Lemma~\ref{8}d, $(\langle\Gamma\rangle,X)$ is pointwise $\Gamma$-recurrent of type I.
Since $\overline{Gx}=\overline{\langle \Gamma\rangle x}$ for all $x\in X$, $x$ is a.p under $(G,X)$ iff $x$ is a.p under $(\langle \Gamma\rangle,X)$. Then
i)$\Leftrightarrow$ii)$\Leftrightarrow$iii)$\Leftrightarrow$iv)$\Leftrightarrow$iv$^\prime$) by Theorem~\ref{4}a, where
\begin{enumerate}[iv$^\prime$)]
\item $\mathscr{O}_{G,X}\colon X\rightarrow 2^X$ is upper semi-continuous.
\end{enumerate}
v)$\Rightarrow$ii): This is evident by 3) of v).

\noindent
iv$^\prime$)$\Rightarrow$v): Assume iv$^\prime$). Define $Y=X/R_o$ and let $\rho\colon X\rightarrow Y$ be the canonical quotient mapping. Then $\rho$ is closed (cf.~\cite[Thm.~3.12]{K}), and moreover, by iii) it is easy to see that $Y$ is a Hausdorff space. Let $y\in Y$ and $U\in\mathfrak{N}_y$. Set $V=\rho^{-1}U$ and $x\in\rho^{-1}y$. Then $\overline{Gx}=\rho^{-1}y\subseteq V$. Further, there exists a compact open neighborhood $W$ of $\overline{Gx}$. By iii), $W_G^*$ is compact open $\rho$-saturated with $x\in W_G^*$. Thus $\rho W_G^*\subseteq U$ is a clopen compact neighborhood of $y$ in $Y$. This shows that $Y$ is locally compact zero-dimensional. Thus $(G,Y)$ is a flow having the properties 1), 2) and 3).

The proof of Theorem~\ref{8}f is thus completed.
\end{proof}
\end{se}

\begin{rem}
The $\Gamma$-recurrence of type I (Def.~\ref{2}) is conceptually dependent of the generating set $\Gamma$ of $G$. So it should be interested to generalize Proposition~\ref{7}f from $G=\mathbb{Z}$ to a more general case.
\end{rem}
\medskip

\noindent
\textbf{Acknowledgments. }
This work was supported by National Natural Science Foundation of China (Grant No. 11790274) and PAPD of Jiangsu Higher Education Institutions.

\end{document}